\documentclass[11pt,a4paper]{article}
\usepackage[T1]{fontenc}
\usepackage[utf8]{inputenc}
\usepackage[english]{babel}
\usepackage{lmodern,microtype,amsmath,amssymb,amsthm,mathtools}
\usepackage[a4paper,margin=2.5cm]{geometry}
\usepackage{enumitem,needspace}
\usepackage{mathrsfs}
\usepackage[numbers]{natbib}
\usepackage[hidelinks]{hyperref}
\newcommand{\resultcite}[2]{[#1,~\citealp{#2}]}
\newtheorem{theorem}{Theorem}[section]
\newtheorem{proposition}[theorem]{Proposition}
\newtheorem{corollary}[theorem]{Corollary}
\newtheorem{lemma}[theorem]{Lemma}
\theoremstyle{remark}
\newtheorem{remark}[theorem]{Remark}
\newcommand{\B}{\mathcal B}
\newcommand{\K}{\mathcal{K}}
\newcommand{\Hq}{\mathcal{H}_0}
\newcommand{\sg}{\sigma}

\DeclareMathOperator{\codim}{codim}
\newcommand{\Ran}{\mathcal{R}}
\newcommand{\Nul}{\mathcal{N}}
\DeclareMathOperator{\acc}{acc}
\setlist[enumerate]{label=(\roman*),leftmargin=2.2em,itemsep=.25em,topsep=.3em}
\title{Algebraic Direct Sums and Generalized Drazin Invertibility}

\author{Zakariae Aznay \qquad Fadil Chabbabi\\[-.5mm]
\small Laboratory (A.F.N.L.A), Department of Mathematics, Faculty of Sciences,\\[-.5mm]
\small Abdelmalek Essaâdi University, Tetouan, Morocco\\[-.5mm]
\small \texttt{aznayzakariae@gmail.com} \quad
\texttt{f.chabbabi@uae.ac.ma}}
\date{}

\hypersetup{
  pdftitle={Algebraic Direct Sums and Generalized Drazin Invertibility},
  pdfauthor={Zakariae Aznay and Fadil Chabbabi}
}

\begin{document}
\maketitle

\begin{abstract}
Let $T$ be a bounded linear operator on a complex Banach space $X$.
We prove that the algebraic decomposition
\[
X=\K(T)\oplus \Hq(T)
\]
forces both summands to be closed and hence characterizes generalized
Drazin invertibility. This improves the classical characterizations
of Mbekhta~\resultcite{Theorem~1.6}{MbekhtaKato}, Schmoeger~\resultcite{Theorem~4}{Schmoeger}, and Daji\'c and
Koliha~\resultcite{Theorem~6.7}{DajicKoliha}, by removing their closedness
assumptions. We also show that $T$ is generalized Drazin invertible
if and only if
\[
\dim\bigl(\K(T)\cap \Hq(T)\bigr)<\infty
\quad\text{and}\quad
\codim\bigl(\K(T)+\Hq(T)\bigr)<\infty.
\]
For ordinary Drazin invertibility, we obtain analogous finite-defect
characterizations involving the hyperkernel and the hyperrange.
We also give a characterization using pointwise nilpotence on an invariant
summand. Finally, we consider arbitrary invariant algebraic decompositions
and give a counterexample when the inverse of the restriction to
one summand is not assumed bounded.
\end{abstract}



\section{Introduction}
Let $X$ be a complex Banach space and let $\B(X)$ denote the algebra of
all bounded linear operators on $X$. For $T\in\B(X)$, its quasinilpotent
part and analytic core are, respectively,
\[
\Hq(T)=\left\{x\in X:\lim_{n\to\infty}\|T^nx\|^{1/n}=0\right\},
\]
\begin{equation}\label{eq:K}
\K(T)=\left\{x\in X:\begin{array}{l}
\text{there are }c>0\text{ and }(x_n)_{n\geq0}\subset X\text{ such that}\\
x_0=x,\quad Tx_{n+1}=x_n\ (n\geq0),\\
\|x_n\|\leq c^n\|x\|\ (n\geq0)
\end{array}\right\}.
\end{equation}
Both are $T$-invariant linear subspaces and need not be closed; moreover, $T\K(T)=\K(T)$ (see \resultcite{Theorem~2(a)}{Schmoeger}). For any
$T$-invariant subspace $M$, we write $T_M$ for the restriction of $T$ to
$M$, endowed with the norm inherited from $X$. We use $X=M\oplus N$ to
mean the algebraic direct sum $X=M+N$ and $M\cap N=\{0\}$, without
assuming the summands closed. When both summands are closed, this is a
topological direct sum. We denote by $c_0(X)$ the complex Banach space
of all sequences $(x_n)_{n\geq0}$ in $X$ such that $\|x_n\|\to0$,
endowed with the norm $\|(x_n)\|_\infty=\sup_{n\geq0}\|x_n\|$.
We denote by $\Ran(T)$ and $\Nul(T)$ the range and kernel of $T$,
respectively. The hyperkernel and hyperrange of $T$ are defined by
\[
\Nul^\infty(T)=\bigcup_{n\geq1}\Nul(T^n),\qquad
\Ran^\infty(T)=\bigcap_{n\geq1}\Ran(T^n).
\]
Recall that $T$ is Drazin invertible if there are closed $T$-invariant
subspaces $M,N$ such that $X=M\oplus N$, $T_M$ is invertible and
$T_N$ is nilpotent. Equivalently, the kernels and ranges of the powers
of $T$ both stabilize after finitely many steps; see
\resultcite{Theorem~3}{Schmoeger}, with the invertible case added separately.

Recall that $T$ is generalized Drazin invertible (or Koliha--Drazin
invertible) if it admits a topological decomposition $X=M\oplus N$
into closed $T$-invariant subspaces such that $T_M$ is invertible and
$T_N$ is quasinilpotent. By Koliha's spectral characterization, this is equivalent to
$0\notin\acc\sigma(T)$~\resultcite{Theorems~4.2 and~7.1}{Koliha}. Mbekhta studied the analytic core and the quasinilpotent part
and showed, for operators on Hilbert spaces, that $T$ is generalized
Drazin invertible if and only if $X=\K(T)\oplus\Hq(T)$ is a
topological direct sum~\resultcite{Theorem~1.6}{MbekhtaKato}. See also~\cite{Mbekhta}
for further properties of these subspaces. Schmoeger obtained an
algebraic version assuming that
$\K(T)$ is closed~\resultcite{Theorem~4}{Schmoeger}, and Daji\'c and Koliha showed that
the closedness of either summand suffices~\resultcite{Theorem~6.7}{DajicKoliha}.
The formulations concerning an isolated spectral point are understood
with the invertible case added separately.
We first remove this additional closedness assumption. We then give a
characterization involving the dimension of $\K(T)\cap \Hq(T)$ and the
codimension of $\K(T)+\Hq(T)$. We also give new characterizations of Drazin invertibility using the
hyperkernel and the hyperrange. Finally, we examine algebraic
decompositions into arbitrary invariant subspaces and show, by an example
on a separable Hilbert space, why the boundedness of an inverse restriction
matters in the generalized Drazin case. We also derive a Drazin
characterization for pointwise nilpotent summands.

\section{Preliminary results on the quasinilpotent part and the analytic core}\label{sec:structure}
We begin with a result concerning the analytic core which will be useful in the next section. For every integer $m\geq1$ define
\begin{equation}\label{eq:Ym}
Y_m=\bigl\{y=(y_n)_{n\geq0}\in c_0(X):mTy_{n+1}=y_n\ (n\geq0)\bigr\}
\end{equation}
and let
\begin{equation}\label{eq:Em}
E_m:Y_m\to X,\qquad E_m(y)=y_0,\qquad \mathcal{K}_m=E_m(Y_m).
\end{equation}

\begin{proposition}\label{prop:Banach-images}
Each $Y_m$ is a Banach space (separable when $X$ is separable), the map $E_m$ is linear and bounded, and
\[
\K(T)=\bigcup_{m\geq1}\mathcal{K}_m.
\]
\end{proposition}
\begin{proof}
For fixed $n$, the map $y\mapsto mTy_{n+1}-y_n$ is linear and
bounded from $c_0(X)$ to $X$, since
$\|mTy_{n+1}-y_n\|\leq (m\|T\|+1)\|y\|_\infty$.
Hence $Y_m$, the intersection of its kernels over $n\geq0$,
is closed in $c_0(X)$. It is therefore a Banach space. If $X$ is
separable, then $c_0(X)$ is separable, and hence $Y_m$ is separable.
Moreover, $E_m$ is linear and bounded, since
$\|E_m(y)\|=\|y_0\|\leq\|y\|_\infty$ for every $y\in Y_m$.

Take $x\in \mathcal{K}_m$. If $x=0$ it lies in $\K(T)$. Otherwise choose $y\in Y_m$
with $y_0=x$; write $M=\sup_n\|y_n\|<\infty$. Set $x_n=m^ny_n$. Then
$x_0=x$ and
\[
Tx_{n+1}=m^{n+1}Ty_{n+1}=m^ny_n=x_n.
\]
Choose $c=m\max\{1,M/\|x\|\}$. For $n\geq1$, since $M\geq\|x\|$,
we have
\[
\|x_n\|\leq m^nM\leq c^n\|x\|;
\]
the bound at $n=0$ is equality. This proves $x\in \K(T)$.

Conversely, if $x\in \K(T)$, choose a backward chain $x_n$ and $c>0$ as in
\eqref{eq:K}. Take an integer $m>c$ and put $y_n=m^{-n}x_n$. Then
\[
\|y_n\|\leq(c/m)^n\|x\|\longrightarrow0,\qquad mTy_{n+1}=y_n.
\]
Thus $y\in Y_m$ and $E_m(y)=x$, proving the converse inclusion.
\end{proof}
\begin{lemma}\label{lem:pk} 
Let $X$ be a complex Banach space and $T\in\mathcal{B}(X)$. For $k\ge1$ and $x\in X$, put
\[
p_{k}(x)=\sup_{r\ge0} k^r\|T^rx\|\in[0,\infty].
\]
Then $x\in \Hq(T)$ if and only if $p_k(x)<\infty$ for every $k\ge1$.
\end{lemma}
\begin{proof}
Assume that $x\in \Hq(T)$. Then, for every $k\ge1$, we have $\|T^rx\|\le(2k)^{-r}$ for sufficiently large $r$. Hence $p_k(x)<\infty$.
Conversely, for $x\ne0$ and $r\ge1$,
\[
\|T^rx\|^{1/r}\le k^{-1}p_{k}(x)^{1/r}.
\]
Since $0<\|x\|\le p_k(x)<\infty$, we get
$\limsup_{r\to\infty}\|T^rx\|^{1/r}\le1/k$ for every $k\ge1$.
Thus $x\in \Hq(T)$. The case $x=0$ is clear.
\end{proof}
\begin{proposition}\label{prop:Frechet}
Let $X$ be a complex Banach space and $T\in\B(X)$. For every integer $k\ge1$ define, on $\Hq(T)$,
\[
 p_k(x)=\sup_{r\ge0}k^r\|T^rx\|.
\]
Then $(p_k)_{k\ge1}$ is an increasing family of norms which defines a Fr\'echet topology on $\Hq(T)$. Furthermore,
\[
 \|x\|\le p_1(x),\qquad p_k(Tx)\le\frac1k p_k(x).
\]
Moreover, when $X$ is separable, this Fr\'echet space is separable.
\end{proposition}
\begin{proof}
By Lemma~\ref{lem:pk}, $p_k(x)<\infty$ for every $x\in\Hq(T)$ and every $k$. The triangle inequality and homogeneity are clear. Since $p_k(x)\ge\|x\|$, $p_k$ is a norm. Moreover, $p_k\le p_{k+1}$ for every $k\ge1$.

This topology is metrizable by the distance
\[
 d(x,y)=\sum_{k=1}^\infty2^{-k}\min\{1,p_k(x-y)\}.
\]
Let us show that $\Hq(T)$ is complete. Let $(x_j)_{j\ge1}$ be a Cauchy sequence for every $p_k$. Since $\|x\|\le p_1(x)$, it is Cauchy in the Banach space $X$. Hence $x_j\to x$ in $X$ for some $x\in X$. Fix an integer $k$ and $\varepsilon>0$. Choose $J$ so that
\[
 p_k(x_j-x_l)\le\varepsilon\qquad(j,l\ge J).
\]
For each fixed $r\ge0$ and $j\ge J$, we have
\[
 k^r\|T^r(x_j-x_l)\|\le\varepsilon\qquad(l\ge J).
\]
Since $T^r$ is bounded and $x_l\to x$ in $X$, letting $l\to\infty$ gives
\[
 k^r\|T^r(x_j-x)\|\le\varepsilon\qquad(r\ge0).
\]
Taking the supremum over $r$, we obtain
$p_k(x_j-x)\le\varepsilon$. In particular $p_k(x)\le p_k(x_j)+p_k(x_j-x)<\infty$. By Lemma~\ref{lem:pk}, $x\in\Hq(T)$ and $x_j\to x$ for every $p_k$. Hence $\Hq(T)$ is a Fr\'echet space.

Finally,
\[
 p_k(Tx)=\sup_{r\ge0}k^r\|T^{r+1}x\|
 =\frac1k\sup_{s\ge1}k^s\|T^sx\|
 \le\frac1k p_k(x).
\]
Assume now that $X$ is separable. Then $c_0(X)$ is separable,
and so is the countable product $\prod_{k\geq1}c_0(X)$. Consider the map
\[
 \Phi:\Hq(T)\longrightarrow\prod_{k\geq1}c_0(X),\qquad
 \Phi(x)=\bigl((k^rT^rx)_{r\geq0}\bigr)_{k\geq1}.
\]
For every $x\in\Hq(T)$ and $k\geq1$, we have
$k^r\|T^rx\|\leq 2^{-r}p_{2k}(x)\to0$, so $\Phi$ is well defined.
Moreover, $\Phi$ is linear, and, denoting by $\pi_k$ the projection
onto the $k$-th factor, we have
\[
 \|\pi_k(\Phi(x))\|_\infty=p_k(x)
 \qquad(x\in\Hq(T),\ k\geq1).
\]
Thus $\Phi$ is continuous and is a topological embedding. Since every
subspace of a separable metrizable space is separable, we conclude
that $\Hq(T)$ is separable.
\end{proof}
\begin{theorem}\label{thm:finitecodim}
Let $X$ be a complex Banach space and $T\in\B(X)$. If
\[
 \codim\big(\Ran(T)+\Hq(T)\big)<\infty,
\]
then $\Ran(T)+\Hq(T)$ is closed in $X$.
\end{theorem}
\begin{proof}
Put $L=\Ran(T)+\Hq(T)$. Since $\codim L<\infty$, there exists a finite-dimensional subspace $F$ of $X$ such that
\[
 X=L\oplus F\quad\text{algebraically}.
\]
Endow $\Hq(T)$ with the Fr\'echet topology of Proposition~\ref{prop:Frechet}. Since $F$ is finite-dimensional,
\[
 E=X\times\Hq(T)\times F
\]
is a Fr\'echet space. Define
\[
 \Psi:E\longrightarrow X,\qquad \Psi(x,h,f)=Tx+h+f.
\]
Clearly, $\Psi$ is linear and surjective. Moreover, $\Psi$ is continuous.
By the open mapping theorem \resultcite{Theorem~2.11}{Rudin}, $\Psi$ is open.
Moreover,
\[
 \Psi^{-1}(L)=X\times\Hq(T)\times\{0\}
\]
is closed in $E$. Since $\Psi$ is an open surjection, it is a quotient map. Therefore
$L$ is closed if and only if $\Psi^{-1}(L)$ is closed.
Hence $L$ is closed in $X$.
\end{proof}

\begin{lemma}\label{lem:preimage}
Put $L=\K(T)+\Hq(T)$. For every $\lambda\in\mathbb C$,
\[
 (T-\lambda I)^{-1}(L)=L,
\]
where the left-hand side denotes the preimage of $L$.
\end{lemma}
\begin{proof}
The inclusion $L\subseteq (T-\lambda I)^{-1}(L)$ follows from the invariance of $\K(T)$ and $\Hq(T)$. Conversely, suppose
\[
 (T-\lambda I)x=k+h,\qquad k\in\K(T),\ h\in\Hq(T).
\]
If $\lambda=0$, choose $k_1\in\K(T)$ with $Tk_1=k$. Then $T(x-k_1)=h$, and $T^{-1}\Hq(T)=\Hq(T)$ gives $x\in L$.

Assume $\lambda\ne0$. First, there is $u\in\Hq(T)$ such that $(T-\lambda I)u=h$. Indeed, put
\[
 u=-\sum_{n=0}^{\infty}\lambda^{-n-1}T^nh.
\]
The series converges in $X$ since $h\in\Hq(T)$. Moreover, for every $r>0$ choose an integer $s>\max\{r,|\lambda|^{-1}\}$. Writing $p_s(h)=\sup_{j\ge0}s^j\|T^jh\|$, we obtain, for $j\ge0$,
\[
 r^j\|T^ju\|\le \frac{p_s(h)}{|\lambda|}
 \Big(\frac rs\Big)^j
 \sum_{n=0}^{\infty}(|\lambda|s)^{-n}<\infty.
\]
Thus $u\in\Hq(T)$. Set $z=x-u$. Then $(T-\lambda I)z=k$. Choose an exponentially bounded backward chain $k_0=k$, $Tk_{n+1}=k_n$. Define $z_0=z$ and
\[
 z_n=\lambda^{-n}z-\sum_{j=1}^n\lambda^{-(n-j+1)}k_j\qquad(n\ge1).
\]
A direct calculation gives $Tz_n=z_{n-1}$; the bound on the $k_j$ shows that $(z_n)$ has at most exponential growth. Hence $z\in\K(T)$ and $x=z+u\in L$.
\end{proof}
\begin{lemma}\label{lemma:corecodim}
Let $X$ be a complex Banach space and let $T\in\B(X)$. Then:
\begin{enumerate}
\item $\dim\bigl(\K(T)\cap \Hq(T)\bigr)<\infty$ if and only if
$\K(T)\cap \Hq(T)=\{0\}$.
\item $\codim\bigl(\K(T)+\Hq(T)\bigr)<\infty$ if and only if
$\K(T)+\Hq(T)=X$.
\end{enumerate}
\end{lemma}
\begin{proof}
\begin{enumerate}
\item Put $E=\K(T)\cap \Hq(T)$. We first show that $T(E)=E$.
For $y\in E$, choose $x\in \K(T)$ such that $Tx=y$, using
$T(\K(T))=\K(T)$. As $y\in \Hq(T)$ and
$T^{-1}(\Hq(T))=\Hq(T)$, we obtain $x\in E$. Hence $T(E)=E$.
If $E$ were nonzero and finite-dimensional, the restriction $T_E$
would be bijective and would have a nonzero eigenvalue. This
contradicts $E\subseteq \Hq(T)$. Thus $\dim E$ is either zero or
infinite.

\item Put $L=\K(T)+\Hq(T)$ and assume $\codim L<\infty$. Since
$L$ is $T$-invariant, the formula
\[
 \widetilde T(x+L)=Tx+L
\]
defines an algebraic linear map on $V=X/L$.
Suppose $V\ne\{0\}$. Since $V$ is a finite-dimensional complex
vector space, there are an eigenvalue $\lambda_0$ of $\widetilde T$
and a nonzero vector $v=x+L$ satisfying
\[
 (\widetilde T-\lambda_0 I)v=0.
\]
Consequently, $x\notin L$ and $(T-\lambda_0 I)x\in L$.
By Lemma~\ref{lem:preimage}, we obtain $x\in L$, a contradiction.
Hence $V=\{0\}$ and $L=X$.
\end{enumerate}
\end{proof}

\section{The canonical algebraic decomposition}
\begin{theorem}\label{thm:main}
Let $X$ be a complex Banach space and $T\in\B(X)$. The following assertions are equivalent:
\begin{enumerate}
\item $T$ is generalized Drazin invertible;
\item $X=\K(T)\oplus\Hq(T)$;
\item $X=\K(T)\oplus\Hq(T)$ is a topological direct sum.
\end{enumerate}
In particular, the algebraic identity in \emph{(ii)} implies that $\K(T)$ and $\Hq(T)$ are closed and that the corresponding algebraic projection is bounded.
\end{theorem}
\begin{proof}
The equivalence of (i) and (iii) follows from \resultcite{Theorem~6.7}{DajicKoliha}. We only prove (ii)$\Rightarrow$(iii).  By Proposition~\ref{prop:Frechet}, $\mathcal H_0(T)$ is a Fr\'echet space
continuously embedded in $X$. We use the Banach spaces $Y_m$ and the
maps $E_m$ defined in \eqref{eq:Ym}--\eqref{eq:Em}; by
Proposition~\ref{prop:Banach-images},
\[
 \K(T)=\bigcup_{m\geq1}E_m(Y_m).
\]

Now $F_m=Y_m\times \mathcal H_0(T)$ is Fr\'echet, and
\[
 \Phi_m:F_m\longrightarrow X,\qquad \Phi_m(y,h)=E_m(y)+h
\]
is linear and continuous, since $E_m$ is bounded and $\|h\|\leq p_1(h)$. By Proposition~\ref{prop:Banach-images} and the hypothesis,
\[
 X=\mathcal K(T)+\mathcal H_0(T)=\bigcup_{m\geq1}\Ran(\Phi_m).
\]
Since $X$ is a Baire space, some $\Ran(\Phi_{m_0})$ is nonmeagre. The nonmeagre-range form of the open mapping theorem \resultcite{Theorem~2.11}{Rudin} implies that $\Phi_{m_0}$ is onto and open.

Let $P:X\to X$ be the algebraic projection onto $\Hq(T)$ along
$\K(T)$, defined by $P(k+h)=h$ for every $k\in\K(T)$ and
$h\in\Hq(T)$. For every $(y,h)\in F_{m_0}$, we have
\[
 P\Phi_{m_0}(y,h)=h.
\]
If $U$ is open in $X$, then
\[
 P^{-1}(U)=\Phi_{m_0}\bigl(\{(y,h)\in F_{m_0}:h\in U\}\bigr).
\]
The set inside the parentheses is open, since $\Hq(T)$ is continuously
embedded in $X$, and $\Phi_{m_0}$ is open. Thus $P$ is continuous.
As $P^2=P$,
\[
 \K(T)=\Nul(P),\qquad \Hq(T)=\Nul(I-P)
\]
are closed, as claimed.
\end{proof}

\begin{corollary}\label{cor:category}
For every $T\in\B(X)$, either $\K(T)+\Hq(T)=X$ or $\K(T)+\Hq(T)$ is meagre in $X$. Consequently, $T$ is generalized Drazin invertible if and only if
\[
 \K(T)\cap\Hq(T)=\{0\},\qquad
 \K(T)+\Hq(T)\text{ is nonmeagre in }X.
\]
\end{corollary}
\begin{proof}
If the sum is nonmeagre, the union of the ranges of the maps $\Phi_m$ in the preceding proof is nonmeagre. Hence one of those maps is onto by \resultcite{Theorem~2.11}{Rudin}, and the sum equals $X$. The last assertion follows from Theorem~\ref{thm:main}.
\end{proof}

\begin{proposition}\label{thm:defects}
Let $X$ be a complex Banach space and $T\in\B(X)$. The following assertions are equivalent:
\begin{enumerate}
\item $T$ is generalized Drazin invertible;
\item $\dim(\K(T)\cap\Hq(T))<\infty$ and
$\codim(\K(T)+\Hq(T))<\infty$.
\end{enumerate}
\end{proposition}
\begin{proof}
If (ii) holds, Lemma~\ref{lemma:corecodim} gives $\K(T)\cap\Hq(T)=\{0\}$ and
$\K(T)+\Hq(T)=X$. Hence $X=\K(T)\oplus\Hq(T)$, and
Theorem~\ref{thm:main} yields (i). The converse follows from the
canonical topological decomposition.
\end{proof}

\section{Consequences for Drazin invertibility}
The preceding results yield the following characterization without any new assumption of closedness.

\begin{corollary}\label{cor:drazin-from-gdrazin}
For $T\in\B(X)$ the following assertions are equivalent:
\begin{enumerate}
\item $T$ is Drazin invertible;
\item $X=\K(T)\oplus\Hq(T)$ and $\Hq(T)=\Nul^\infty(T)$;
\item $\dim(\K(T)\cap\Hq(T))<\infty$,
$\codim(\K(T)+\Hq(T))<\infty$, and
$\Hq(T)=\Nul^\infty(T)$.
\end{enumerate}
In this case $\Hq(T)=\Nul(T^d)$ and
$\K(T)=\Ran(T^d)$ for all sufficiently large $d$.
\end{corollary}
\begin{proof}
By Theorem~\ref{thm:main}, (ii) implies that $T$ is generalized
Drazin invertible and $\Hq(T)$ is closed. Since
$\Hq(T)=\Nul^\infty(T)=\bigcup_{n\geq1}\Nul(T^n)$,
the Baire theorem, applied to the closed subspaces
$\Nul(T^n)$ of the Banach space $\Hq(T)$, gives
$\Hq(T)=\Nul(T^d)$ for some $d$. Thus the quasinilpotent
restriction is nilpotent, and $T$ is Drazin invertible.
The converse follows from the usual Drazin decomposition;
the equivalence of (ii) and (iii) follows at once from
Proposition~\ref{thm:defects}.
\end{proof}

The hyperkernel can also replace the quasinilpotent part
in the canonical algebraic decomposition.
\begin{corollary}\label{cor:drazin-core}
An operator $T\in\B(X)$ is Drazin invertible if and only if
\[
X=\K(T)\oplus\Nul^\infty(T)
\]
algebraically. Equivalently, $T$ is Drazin invertible if and only if
\[
\dim\bigl(\K(T)\cap\Nul^\infty(T)\bigr)<\infty,
\qquad
\codim\bigl(\K(T)+\Nul^\infty(T)\bigr)<\infty.
\]
Moreover, each of these two defects is either zero or infinite.
\end{corollary}
\begin{proof}
Suppose $X=\K(T)\oplus\Nul^\infty(T)$. The same Baire argument
as in Theorem~\ref{thm:main}, using
$\K(T)=\bigcup_{m\geq1}E_m(Y_m)$ and
$\Nul^\infty(T)=\bigcup_{n\geq1}\Nul(T^n)$, gives
$X=E_m(Y_m)+\Nul(T^n)$ for some $m,n$. Uniqueness of the
algebraic coordinates yields $\Nul^\infty(T)=\Nul(T^n)$.
Since $T\K(T)=\K(T)$, we then have
$\Ran(T^n)=\K(T)$. Thus the kernels and ranges stabilize, and $T$ is Drazin invertible.
The converse is immediate. Finally, the arguments of
Lemma~\ref{lemma:corecodim} apply verbatim with
$\Nul^\infty(T)$ in place of $\Hq(T)$: indeed,
$T^{-1}\Nul^\infty(T)=\Nul^\infty(T)$, and for
$\lambda\ne0$ the equation $(T-\lambda I)u=h$ with
$T^qh=0$ admits the solution
$u=-\sum_{j=0}^{q-1}\lambda^{-j-1}T^jh\in\Nul^\infty(T)$.
Thus both finite defects must vanish, and the preceding equivalence
applies.
\end{proof}

In \cite{AznayAlmost}, it was proved that $T$ is Drazin invertible
if and only if
\[
X=\Ran^\infty(T)\oplus\Nul^\infty(T),
\]
without assuming either subspace closed. The following theorem extends
this characterization to finite-dimensional intersection and finite
codimension of the sum.
\begin{theorem}\label{thm:drazin-hyper-defects}
Let $T\in\B(X)$. Then $T$ is Drazin invertible if and only if
\[
\dim\bigl(\Ran^\infty(T)\cap\Nul^\infty(T)\bigr)<\infty
\quad\text{and}\quad
\codim\bigl(\Ran^\infty(T)+\Nul^\infty(T)\bigr)<\infty.
\]
\end{theorem}
\begin{proof}
Necessity follows from the Drazin decomposition. To prove sufficiency,
put $U=\Ran^\infty(T)$, $B=\Nul^\infty(T)$ and $L=U+B$,
and choose a finite-dimensional algebraic complement $F$ of $L$ in $X$.
The space
\[
Y=\{(x_j)_{j\geq0}\in X^{\mathbb N_0}:T^jx_j=x_0\ (j\geq1)\}
\]
is a closed subspace of the Fr\'echet space $X^{\mathbb N_0}$.
Its continuous zeroth-coordinate map $E:Y\to X$ has range $U$.
For each $n\geq1$, consider the continuous linear map
\[
\Phi_n:Y\times\Nul(T^n)\times F\longrightarrow X,
\qquad \Phi_n(y,z,f)=E(y)+z+f.
\]
Since $X=\bigcup_{n\geq1}\Ran(\Phi_n)$, the Baire theorem
implies that some $\Ran(\Phi_n)$ is nonmeagre. By the open mapping
theorem \resultcite{Theorem~2.11}{Rudin}, this $\Phi_n$ is onto.
As $F\cap L=\{0\}$, it follows that
\[
L=U+\Nul(T^n),\qquad B=(U\cap B)+\Nul(T^n).
\]
The finite-dimensional subspace $U\cap B$ is contained in
$\Nul(T^p)$ for some $p\geq n$. Hence
\[
B=\Nul(T^p)=\Nul(T^{p+1}).
\]
Moreover, $\Ran(T^p)\cap B=\{0\}$: if $T^px\in B$,
then $T^{2p}x=0$, so $x\in B=\Nul(T^p)$.

For $k\geq p$, the subspaces $\Ran(T^k)+B$ form a decreasing
sequence containing $L$. Since $X/L$ is finite-dimensional,
this sequence stabilizes. Thus, for some $q\geq p$,
\[
\Ran(T^q)+B=\Ran(T^{q+1})+B.
\]
If $x\in\Ran(T^q)$, write $x=y+b$ with
$y\in\Ran(T^{q+1})$ and $b\in B$. Then
$b=x-y\in\Ran(T^q)\cap B=\{0\}$. Consequently,
$\Ran(T^q)=\Ran(T^{q+1})$.
The kernels and ranges of the powers of $T$ therefore stabilize,
so $T$ is Drazin invertible.
\end{proof}

\begin{remark}\label{rem:hyper-exception}
Unlike $\K(T)\cap\Hq(T)$, the intersection
$\Ran^\infty(T)\cap\Nul^\infty(T)$ may have positive finite
dimension. We give an example for which this dimension is one.

Let $H$ be a Hilbert space with orthonormal basis
\[
\{e_0\}\cup\{e_{j,k}:j\geq1,\ 1\leq k\leq j\},
\]
and define
\[
Te_0=0,\qquad Te_{j,1}=2^{-j}e_0,\qquad
Te_{j,k}=e_{j,k-1}\quad(2\leq k\leq j).
\]
For $x=a_0e_0+\sum_{j,k}a_{j,k}e_{j,k}$, we have
\[
Tx=\Bigl(\sum_{j\geq1}2^{-j}a_{j,1}\Bigr)e_0
  +\sum_{j\geq1}\sum_{k=2}^j a_{j,k}e_{j,k-1}.
\]
Since $\sum_{j\geq1}4^{-j}=1/3$, the Cauchy--Schwarz
inequality gives $\|Tx\|\leq\|x\|$. Thus $T$ is bounded.

For every $n\geq1$, we have
\[
T^n(2^ne_{n,n})=e_0.
\]
Hence $e_0\in\Ran^\infty(T)$. Conversely, fix $j\geq1$
and $1\leq k\leq j$. For $n>j-k$, the coefficient of
$e_{j,k}$ in $T^nx$ is zero for every $x\in H$.
Therefore every vector of $\Ran^\infty(T)$ has zero
coefficient at each $e_{j,k}$. It follows that
\[
\Ran^\infty(T)=\mathbb Ce_0.
\]
Since $Te_0=0$, we also have $e_0\in\Nul^\infty(T)$.
Consequently,
\[
\Ran^\infty(T)\cap\Nul^\infty(T)=\mathbb Ce_0.
\]
Thus the zero-or-infinity alternative for
$\dim(\K(T)\cap\Hq(T))$ does not extend to
$\dim(\Ran^\infty(T)\cap\Nul^\infty(T))$.
\end{remark}

\begin{corollary}\label{cor:drazin-nilpotent-split}
Suppose $X=M\oplus N$ algebraically, where $M,N$ are
$T$-invariant, $T_M$ is bijective, and $T_N^d=0$ for some $d\geq1$.
Then $T$ is Drazin invertible, without assuming $T_M^{-1}$
bounded. Moreover,
\[
M=\Ran(T^d)=\K(T),\qquad
N=\Nul(T^d)=\Nul^\infty(T),
\]
and both summands are closed.
\end{corollary}
\begin{proof}
Since $T^d(M)=M$ and $T^d(N)=\{0\}$, we have
$\Ran(T^d)=M$ and $\Nul(T^d)=N$. The conclusion follows
from the stabilization of the kernels and ranges of the powers of $T$.
\end{proof}

\Needspace{13\baselineskip}
\section{Other algebraic decompositions}
For a possibly nonclosed invariant subspace $M$, all operator norms below are taken for the norm inherited from $X$.

\begin{proposition}\label{prop:pointwise}
For $T\in\B(X)$ the following are equivalent:
\begin{enumerate}
\item $T$ is generalized Drazin invertible;
\item there are $T$-invariant subspaces $M,N$ such that $X=M\oplus N$, $T_M$ is bijective with bounded inverse, and
\[
 \lim_{n\to\infty}\|T^nx\|^{1/n}=0 \qquad \text{ for all } x\in N.
\]
 \item there are $T$-invariant subspaces $M,N$ such that $X=M\oplus N$, $T_M$ is bijective with bounded inverse, and
\[
 \lim_{n\to\infty}\|T_N^n\|^{1/n}=0.
\]
\end{enumerate}
In this situation $M=\K(T)$ and $N=\Hq(T)$; hence both are closed and $T_N$ is quasinilpotent.
\end{proposition}
\begin{proof}
Only (ii)$\Rightarrow$(i) requires proof. Put $F=\overline M$ and
$S=T_M^{-1}$. Since $S$ is bounded in the inherited norm, it extends to
an operator $\widetilde S\in\B(F)$. Passing to the limit in
$T_MS=ST_M=I_M$ gives
$T_F\widetilde S=\widetilde S T_F=I_F$. Thus $T_F$ is invertible.
By (ii), $N\subseteq\Hq(T)$. If $x\in F\cap\Hq(T)$, then
\[
 \|x\|\leq\|T_F^{-1}\|^n\|T^nx\|\qquad(n\geq1),
\]
so $x=0$. As $X=M\oplus N$, we therefore have
\[
 X=F\oplus\Hq(T)
\]
algebraically. Endow $\Hq(T)$ with the Fr\'echet topology of
Proposition~\ref{prop:Frechet}. The addition map
$F\times\Hq(T)\to X$ is a continuous linear bijection between
Fr\'echet spaces. By the open mapping theorem its inverse is continuous,
so both summands are closed. Since $\Hq(T)$ is closed,
$T_{\Hq(T)}$ is quasinilpotent by the uniform boundedness principle.
Thus $T$ is generalized Drazin invertible and, by the canonical
decomposition, $F=\K(T)$.
Finally, write $f=m+n$ for $f\in F$, with $m\in M$ and $n\in N$.
Then $n=f-m\in F\cap\Hq(T)=\{0\}$, so $M=F$.
Similarly, writing $h=m+n$ for $h\in\Hq(T)$ gives
$m=h-n\in F\cap\Hq(T)=\{0\}$, so $N=\Hq(T)$.
\end{proof}
\begin{corollary}\label{cor:idempotent}
For $T\in\B(X)$ the following are equivalent:
\begin{enumerate}
\item $T$ is generalized Drazin invertible;
\item there exists a linear idempotent \(P:X\to X\), not necessarily bounded, commuting with \(T\), such that \(T_{\mathcal N(P)}\) is bijective with bounded inverse and 
\[
 \lim_{n\to\infty}\|T^n(Px)\|^{1/n}=0\qquad(x\in X).
\]
\end{enumerate}
Moreover, in this case, \(P\) is automatically bounded and
\[
\mathcal N(P)=\mathcal K(T),
\qquad
\mathcal R(P)=\mathcal H_0(T).
\]
\end{corollary}
\begin{proof}
Assume that (ii) holds. Since \(P^2=P\) and \(PT=TP\), we have
$
X=\mathcal N(P)\oplus\mathcal R(P),$ where both subspaces are invariant under \(T\). The remaining assumptions allow us to apply the preceding proposition with
\[
M=\mathcal N(P),\qquad N=\mathcal R(P).
\]

Hence \(T\) is generalized Drazin invertible, and
\[
\mathcal N(P)=\mathcal K(T),\qquad
\mathcal R(P)=\mathcal H_0(T).
\]

Consequently, \(P\) is bounded. The converse follows by taking the canonical projection onto \(\mathcal H_0(T)\) along \(\mathcal K(T)\).
\end{proof}

The preceding proposition also gives a characterization of Drazin
invertibility when the restriction to the second summand is only
pointwise nilpotent.

\begin{corollary}\label{cor:drazin-pointwise}
For $T\in\B(X)$ the following assertions are equivalent:
\begin{enumerate}
\item $T$ is Drazin invertible;
\item there are $T$-invariant subspaces $M,N$ such that $X=M\oplus N$,
$T_M$ is bijective, and for every $x\in N$ there exists $n_x\geq1$
with $T^{n_x}x=0$, and at least one of the following holds:
\begin{enumerate}[label=(\alph*),leftmargin=2.1em,itemsep=.1em,topsep=.1em]
\item $T_M^{-1}$ is bounded;
\item $M$ is closed;
\item $N$ is closed.
\end{enumerate}
\end{enumerate}
In each case both summands are closed and
\[
M=\K(T)=\Ran^\infty(T),\qquad
N=\Hq(T)=\Nul^\infty(T).
\]
\end{corollary}
\begin{proof}
Only (ii)$\Rightarrow$(i) requires proof. Assume first that (a) holds.
Since $N\subseteq\Hq(T)$, Proposition~\ref{prop:pointwise} gives
$N=\Hq(T)$, which is closed. As
$N=\bigcup_{n\geq1}\Nul(T_N^n)$, the Baire theorem gives
$T_N^d=0$ for some $d\geq1$. Hence $T$ is Drazin invertible.
If (b) holds, the bounded inverse theorem shows that $T_M^{-1}$
is bounded, and the preceding case applies. Finally, if (c) holds,
the Baire theorem applied to $N$ gives $T_N^d=0$ for some $d$.
Corollary~\ref{cor:drazin-nilpotent-split} then yields (i).
The converse follows from the usual Drazin decomposition.
The identities for $M$ and $N$ follow from
Proposition~\ref{prop:pointwise} in case (a), from case (a) in
case (b), and from Corollary~\ref{cor:drazin-nilpotent-split}
in case (c).
\end{proof}

\begin{remark}\label{rem:drazin-pointwise-example}
The additional assumption in Corollary~\ref{cor:drazin-pointwise} cannot in general
be omitted. Let $T$ be the backward shift on
$X=\ell^2(\mathbb N_0)$, defined by
\[
T(x_0,x_1,x_2,\ldots)=(x_1,x_2,x_3,\ldots),
\]
and put $N=c_{00}$. Clearly, $T_N$ is pointwise nilpotent
and $T(N)=N$. Moreover, for every $\lambda\neq 0$ and
$x\in N$ with $T^kx=0$, we have
\[
(T-\lambda I)\left(
-\sum_{j=0}^{k-1}\frac{T^jx}{\lambda^{j+1}}
\right)=x.
\]
Consequently, $p(T)(N)=N$ for every nonzero polynomial
$p\in\mathbb C[z]$.

We show that $N$ admits an algebraic $T$-invariant
complement in $X$. Consider the pairs $(E,P_E)$, where
$E$ is a $T$-invariant subspace containing $N$ and
$P_E:E\to N$ is linear, with
\[
P_E|_N=I_N,\qquad P_ET=TP_E.
\]
By Zorn's lemma, there exists a maximal such pair
$(E,P_E)$. Suppose that $E\neq X$ and choose
$x\in X\setminus E$. If no nonzero polynomial $q$ satisfies
$q(T)x\in E$, put $v=0$. Otherwise, choose a monic polynomial
$p$ of least degree such that $p(T)x\in E$. Since $p(T)(N)=N$,
we can choose $v\in N$ such that
\[
p(T)v=P_E(p(T)x).
\]
In either case, the formula
\[
P'(y+q(T)x)=P_E(y)+q(T)v
\]
defines a linear map from
$E'=E+\{q(T)x:q\in\mathbb C[z]\}$ into $N$.
Indeed, if $q(T)x\in E$, then $q=0$ in the first case.
In the second case, polynomial division and the minimality of $p$
give $q=pr$ for some polynomial $r$. Hence
\[
P_E(q(T)x)=q(T)v.
\]
Thus $P'$ is well-defined, extends $P_E$, and
commutes with $T$, contradicting maximality.
Therefore $E=X$.

Now put $M=\ker P_E$. Then
\[
X=M\oplus N,\qquad T(M)\subseteq M.
\]
Since $\ker T\subseteq N$, the restriction $T_M$
is injective. Moreover, for every $m\in M$, choose
$x\in X$ with $Tx=m$. Then
\[
T(x-P_Ex)=m,
\qquad x-P_Ex\in M,
\]
which shows that $T_M$ is surjective. Hence
$T_M$ is bijective.

However, for every $|\lambda|<1$,
\[
T(1,\lambda,\lambda^2,\ldots)
=\lambda(1,\lambda,\lambda^2,\ldots).
\]
Thus $0\in\acc\sg(T)$, and
$T$ is not Drazin invertible.

\end{remark}

\section{A counterexample on a separable Hilbert space}
The boundedness of $T_M^{-1}$ in Proposition~\ref{prop:pointwise} cannot in general be omitted, as the following example shows.
\begin{proposition}\label{prop:counter}
There exist a separable Hilbert space $X$, an operator
$T\in\B(X)$ and $T$-invariant subspaces $M,N$
such that
\[
X=M\oplus N,\qquad
T_M\text{ is bijective},\qquad
\lim_{n\to\infty}\|T_N^n\|^{1/n}=0,
\]
although $T$ is not generalized Drazin invertible.
Moreover, $T_M^{-1}$ is unbounded.
Here the direct sum is algebraic.
\end{proposition}

\begin{proof}
Let $H=L^2(0,1;\mathbb C)$ and let $V$ be the Volterra
operator
\[
(Vf)(t)=\int_0^t f(s)\,ds.
\]
For $n\geq1$,
\[
(V^nf)(t)
=\frac1{(n-1)!}\int_0^t(t-s)^{n-1}f(s)\,ds,
\]
so Young's inequality gives
\[
\|V^n\|\leq\frac1{n!}.
\]
Thus $V$ is quasinilpotent. Since $V$ is injective and
quasinilpotent, $p(V)$ is injective for every nonzero
polynomial $p$.

Put
\[
D=C^\infty([0,1]),\qquad
F=\{f\in D:f^{(j)}(0)=0\text{ for every }j\geq0\},
\]
where smooth functions are identified with their classes
in $H$. We first construct an algebraic projection
$R:H\to F$ commuting with $V$.

We have $VF=F$, since $Vf'=f$ for every $f\in F$.
If $\lambda\neq0$ and $g\in F$, then
\[
f(t)=\frac{g(t)}{\lambda}
+\frac1{\lambda^2}
 \int_0^t e^{(t-s)/\lambda}g(s)\,ds
\]
belongs to $F$ and satisfies $(\lambda I-V)f=g$.
Factoring polynomials into linear factors, we obtain
\[
p(V)F=F
\qquad\text{for every nonzero polynomial }p.
\]

By Zorn's lemma, choose a maximal $V$-invariant subspace
$E\subset H$ satisfying $E\cap F=\{0\}$.
Suppose that $F+E\neq H$ and take $x\notin F+E$.
By maximality,
\[
E+\operatorname{span}\{V^jx:j\geq0\}
\]
meets $F$ nontrivially. Hence
\[
p(V)x=f+e
\]
for some nonzero polynomial $p$, $f\in F$ and $e\in E$.
Choose $g\in F$ such that $p(V)g=f$, and put $y=x-g$.
Then $y\notin F+E$ and $p(V)y=e$.
If
\[
h=e_1+q(V)y\in F,\qquad e_1\in E,
\]
then
\[
p(V)h=p(V)e_1+q(V)e\in E\cap F=\{0\}.
\]
Since $p(V)$ is injective, $h=0$.
Thus
\[
E+\operatorname{span}\{V^jy:j\geq0\}
\]
is a strictly larger invariant subspace disjoint from $F$,
a contradiction. Therefore $H=F\oplus E$.
The projection $R$ onto $F$ along $E$ satisfies
\[
R|_F=I_F,\qquad RV=VR.
\]
No continuity of $R$ is asserted.

Set $G=D\cap\ker R$. Since $R(D)\subseteq F$, we have
\[
D=F\oplus G,\qquad VG\subseteq G.
\]
Consider the jet map
\[
J:D\longrightarrow\mathbb C^{\mathbb N_0},
\qquad Jf=(f^{(j)}(0))_{j\geq0}.
\]
The surjectivity of $J$ is Borel's theorem
\resultcite{Theorem~1.2.6}{Hormander}. We recall its proof in this setting.
Given $(c_j)_{j\geq0}$, choose
$\chi\in C_c^\infty(\mathbb R)$ equal to $1$ near $0$.
For each $j\geq1$, choose $\varepsilon_j>0$ sufficiently
small that
\[
\left\|
\frac{c_j}{j!}t^j\chi(t/\varepsilon_j)
\right\|_{C^k([0,1])}
\leq2^{-j}
\qquad(0\leq k<j).
\]
This is possible because, for fixed $j$ and $k<j$,
the displayed norm tends to zero as
$\varepsilon_j\to0$.
With any $\varepsilon_0>0$, the series
\[
f(t)=\sum_{j=0}^{\infty}
\frac{c_j}{j!}t^j\chi(t/\varepsilon_j)
\]
converges in $C^\infty([0,1])$ and satisfies
$f^{(j)}(0)=c_j$ for every $j\geq0$.

Since $\ker J=F$, the restriction $J|_G$ is bijective.
Let
\[
S=(J|_G)^{-1}:\mathbb C^{\mathbb N_0}\longrightarrow G
\]
and define
\[
U(c_0,c_1,\ldots)=(0,c_0,c_1,\ldots).
\]
The identity $JV=UJ$ and the invariance of $G$ give
\[
VS=SU.
\]

We next extend $J$ algebraically to $H$ while preserving
this identity. Define on $D+VH$
\[
\ell_0(f+Vh)=f(0),
\]
and extend $\ell_0$ linearly to $H$.
This is well defined because $f+Vh=0$ implies $f(0)=0$.
Once $\ell_{j-1}$ has been defined on $H$ and agrees with
$f\mapsto f^{(j-1)}(0)$ on $D$, define
\[
\ell_j(f+Vh)=f^{(j)}(0)+\ell_{j-1}(h)
\]
on $D+VH$, and extend it linearly to $H$.
Indeed, if $f+Vh=0$, then $h=-f'\in D$, as an equality
in $H$, and hence
\[
f^{(j)}(0)+\ell_{j-1}(h)=0.
\]
Thus each extension is well defined. The resulting map
\[
\Lambda:H\longrightarrow\mathbb C^{\mathbb N_0},
\qquad
\Lambda h=(\ell_j(h))_{j\geq0},
\]
satisfies
\[
\Lambda|_D=J,\qquad \Lambda V=U\Lambda.
\]

Define
\[
P=R+S\Lambda:H\longrightarrow D.
\]
Then
\[
PV=VR+SU\Lambda=V(R+S\Lambda)=VP.
\]
For $f\in D$, we have $f-Rf\in G$ and
$J(f-Rf)=Jf$, so $SJf=f-Rf$.
Consequently, $Pf=f$ for every $f\in D$.
Thus $P$ is an algebraic projection onto $D$.
Taking $L=\ker P$, we obtain
\[
H=D\oplus L=F\oplus G\oplus L,
\qquad VL\subseteq L.
\]

Let $e_0=(1,0,0,\ldots)\in\mathbb C^{\mathbb N_0}$
and put $g_0=Se_0$.
Take
\[
X=H\oplus_2\ell^2(\mathbb N_0)
\]
and let $B$ be the backward shift on $\ell^2(\mathbb N_0)$.
Define
\[
T(f,a)=(Vf+a_0g_0,Ba),
\qquad a=(a_0,a_1,\ldots).
\]
Clearly $T\in\B(X)$.
Set
\[
M=G\times\ell^2(\mathbb N_0),
\qquad
N=(F\oplus L)\times\{0\}.
\]
Then $X=M\oplus N$, and both subspaces are $T$-invariant.

Under the algebraic identification
\[
M\longrightarrow
\mathbb C^{\mathbb N_0}\times\ell^2(\mathbb N_0),
\qquad (g,a)\longmapsto(Jg,a),
\]
the restriction $T_M$ becomes
\[
(c,a)\longmapsto
\bigl((a_0,c_0,c_1,\ldots),(a_1,a_2,\ldots)\bigr).
\]
Its inverse is
\[
(d,b)\longmapsto
\bigl((d_1,d_2,\ldots),(d_0,b_0,b_1,\ldots)\bigr),
\]
which is well defined for every
$d\in\mathbb C^{\mathbb N_0}$ and
$b\in\ell^2(\mathbb N_0)$.
Hence $T_M$ is bijective.

For $f\in F\oplus L$, we have
$T^n(f,0)=(V^nf,0)$. Therefore
\[
\|T_N^n\|\leq\|V^n\|\leq\frac1{n!},
\]
and consequently
\[
\lim_{n\to\infty}\|T_N^n\|^{1/n}=0.
\]

On the other hand, let $0<|\lambda|<1$ and set
\[
a_\lambda=(1,\lambda,\lambda^2,\ldots),
\qquad
f_\lambda=(\lambda I-V)^{-1}g_0.
\]
Then $a_\lambda\in\ell^2(\mathbb N_0)$ and
\[
T(f_\lambda,a_\lambda)
=\lambda(f_\lambda,a_\lambda).
\]
Thus
\[
\{\lambda\in\mathbb C:0<|\lambda|<1\}
\subseteq\sigma_p(T).
\]
In particular, $0\in\acc\sigma(T)$,
so $T$ is not generalized Drazin invertible.

Finally, put $g_n=V^ng_0$ for $n\geq0$.
These vectors belong to $G$, and
\[
T_M(g_n,0)=(g_{n+1},0).
\]
If $T_M^{-1}$ were bounded, with norm at most $C$, then
\[
0<\|g_0\|
\leq C^n\|g_n\|
\leq\frac{C^n}{n!}\|g_0\|
\qquad(n\geq1),
\]
which is impossible. Hence $T_M^{-1}$ is unbounded.
\end{proof}

\begin{remark}
In the preceding example, $N\subseteq\Hq(T)$, but
$M\cap\Hq(T)\neq\{0\}$. Indeed, $g_0\in G\setminus\{0\}$ and
\[
T^n(g_0,0)=(V^ng_0,0)\qquad(n\geq1).
\]
Since $V$ is quasinilpotent, $(g_0,0)\in M\cap\Hq(T)$.
Thus the algebraic decomposition $X=M\oplus N$ is not
the canonical decomposition involving $\K(T)$ and $\Hq(T)$.
This also shows why the boundedness of $T_M^{-1}$ in
Proposition~\ref{prop:pointwise} cannot be omitted,
even when $\lim_{n\to\infty}\|T_N^n\|^{1/n}=0$.
\end{remark}

\end{document}